\documentclass{amsart}

\usepackage{amsthm,amsfonts}

\usepackage[top=3.5cm,bottom=3.5cm,left=3.5cm,right=3.5cm]{geometry}

\usepackage{amsmath}
\usepackage{graphicx}
\usepackage{float,subcaption}
\usepackage{enumitem}
\usepackage{xcolor}

\newcommand{\loc}[1]{\zeta(#1)}
\newcommand{\scmax}{\Delta}
\newcommand{\oddtreevar}{m}

\usepackage{tikz}
\tikzstyle{vert}=[shape=circle,draw=black,fill=white, inner sep=.5mm]
\usetikzlibrary{decorations.markings,arrows,arrows.meta}
\newcommand{\DAG}{\textrm{DAG-width}}

\newtheorem{theorem}{Theorem}
\newtheorem{corollary}[theorem]{Corollary}
\newtheorem{lemma}[theorem]{Lemma}

\title{The Localization Game on Directed Graphs}

\author[A.\ Bonato]{Anthony Bonato}
\author[R.\ Cushman]{Ryan Cushman}
\author[T.G.\ Marbach]{Trent G.\ Marbach}
\author[B.\ Pittman]{Brittany Pittman}
\address[A1,A2,A3,A4]{Toronto Metropolitan University, Toronto, Canada}

\email[A1]{(A1) abonato@ryerson.ca}
\email[A2]{(A2) ryan.cushman@ryerson.ca}
\email[A3]{(A3) trent.marbach@ryerson.ca}
\email[A4]{(A4) brittany.pittman@ryerson.ca}

\begin{document}
\maketitle

\begin{abstract}
In the Localization game played on graphs, a set of cops uses distance probes to identify the location of an invisible robber. We present an extension of the game and its main parameter, the localization number, to directed graphs. We present several bounds on the localization number of a directed graphs, including a tight bound via strong components, a bound using a linear programming problem on hypergraphs, and bounds in terms of pathwidth and DAG-width. A family of digraphs of order $n$ is given with localization number $(1-o(1))n/2$. We investigate the localization number of random and quasi-random tournaments, and apply our results to doubly regular tournaments, which include Paley tournaments.
\end{abstract}

\section{Introduction}\label{sec:introduction}

Pursuit-evasion games, such as the Localization game and Cops and Robbers, are combinatorial models for detecting or neutralizing
an adversary’s activity on a graph. In such models, pursuers attempt to capture an evader loose on the vertices of a graph. How the players move and the rules of capture depend on which variant is studied. Such games are motivated by foundational topics in computer science, discrete mathematics, and artificial intelligence, such as robotics and network security. For surveys of pursuit-evasion games, see the books \cite{bonato1,bonato2}; see Chapter~5 of \cite{bonato2} for more on the Localization game.

Our focus is on extending the Localization game to the setting of directed graphs. We first review the game played on graphs, as many parts of the game carry over to the directed setting. In the \emph{Localization game}, there are two players playing on a graph, with one player controlling a set of $k$ \emph{cops}, where $k$ is a positive integer, and the second controlling a single \emph{robber}. The game is played over a sequence of discrete time-steps; a \emph{round} of the game is a move by the cops together with the subsequent move by the robber. The robber occupies a vertex of the graph, and when the robber is ready to move during a round, they may move to a neighboring vertex or remain on their current vertex. A move for the cops is a placement of cops on a set of vertices (note that the cops are not limited to moving to neighboring vertices). The players move on alternate time-steps with the robber going first. In each round, the cops occupy a set of vertices $u_1, u_2, \dots , u_k$ and each cop sends out a \emph{cop probe}, which gives their distance $d_i$, where $1\le i \le k$, from $u_i$ to the robber's vertex. The distances $d_i$ are nonnegative integers or may be $\infty.$ Hence, in each round, the cops determine a \emph{distance vector}\index{distance vector} $D=(d_1, d_2, \dots ,d_k)$ of cop probes. Relative to the cops' position, there may be more than one vertex $x$ with the same distance vector. We refer to such a vertex $x$ as a \emph{candidate of} $D$ or simply a \emph{candidate}.  The cops win if they have a strategy to determine, after a finite number of rounds, a unique candidate, at which time we say that the cops {\em capture} the robber. If the robber evades capture, then the robber wins. For a graph $G$, define the \emph{localization number}\index{localization number} of $G$, written $\loc{G}$, to be the least positive integer $k$ for which $k$ cops have a winning strategy.

The Localization game was first introduced for one cop by Seager~\cite{seager1}. In that paper, the robber was not allowed to \emph{backtrack}: they could not move to a vertex containing a cop in the previous round. The game in the present form was first considered in the paper \cite{car}, and subsequently studied in several papers such as~\cite{BBHMP,BHM,BHM1,BK,Bosek2018,nisse2,BDELM}.

The \emph{metric dimension} of a graph $G$, written $\beta(G),$ is equivalent to the minimum number of cops needed in the Localization game so that the cops can win in one round; see \cite{hm,slater}. We have that $\loc{G} \le \beta(G),$ which provides an avenue to apply results on the metric dimension to give bounds on the localization number.  In digraphs, the distance between vertices $u$ and $v$ is the length of a shortest directed path from $u$ to $v$. Note that in digraphs, distance may not be symmetric but the distance relation does satisfy the triangle inequality.  Metric dimension was extending to directed graphs first by Chartrand, Rains, and Zhang in \cite{chartrand-raines-zhang-2000}, and studied further in works such as \cite{ben,feng1,feng2,loz}.

The Localization game on digraphs is played analogously as on graphs, where the robber may only move along directed edges. To simplify notation, for a digraph $G$, we write $\loc{G}$ for its localization number. We follow \cite{ben} and consider non-strongly connected graphs, allowing for cop probes to return $\infty.$

To illustrate the Localization game for directed graphs, consider the digraph $D_3(i)$, which consists of three independent sets of cardinality $i$ labeled $V_0$, $V_1,$ and $V_2,$ each of cardinality $i$, such that
there is an arc from each vertex in $V_j$ to each vertex in $V_{j+1},$ where the indices are taken modulo 3. We then have that $\zeta(D_3(i))=i.$ For a sketch of the proof of this fact, note that $i-1$ cops in $V_j$ may uniquely identify each vertex in that set. The remaining cop may then ``scan'' (or visit each vertex of) both $V_{j+1}$ and $V_{j+2}$ to capture the robber. If fewer than $i$ cops play, then the robber may always find two candidates in a $V_i$ that the cops cannot distinguish.

The present work is the first to consider the Localization game for directed graphs, and we present several results in this direction. In Section~2, we provide bounds on the localization number of digraphs. In Theorem~\ref{thm:ub-sc}, we give a tight bound the localization number by the localization number of its strong components. Using rotation-symmetric tournaments, we provide families of digraphs of order $n$ with localization number $(1-o(1))n/2$ in Theorem~\ref{teg}. We apply a linear programming problem on hypergraphs to give a new bound on the metric dimension in Theorem~\ref{th}, and a lower bound for the localization number is given in Theorem~\ref{dt}. We explore bounds on the localization number in terms of pathwidth and DAG-width in Section~3. We consider the localization of random and quasi-random tournaments in Section~4. Using our results, we give bounds on the localization number of doubly random tournaments, which include Paley tournaments. Our final section includes several open problems.

All digraphs we consider are simple, finite, and \emph{oriented}, in the sense that there is at most one directed edge between a pair of distinct vertices. We refer to a directed edge as an \emph{arc}. For vertices $u$ and $v$ of a graph, if there is an arc between $u$ and $v,$ we denote it by $(u,v).$ For a vertex $u$, we define $N^-(u)$ to be the set out in-neighbors of $u$ and $N^+(u)$ to be the set of out-neighbors of $u.$ We define the in-degree of $u,$ written $\mathrm{deg}^-(u)$, to be $|N^-(u)|$; the out-degree of $u$ is defined analogously. A \emph{sink} is a vertex $u$ such that $N^+(u)=0$ and a \emph{source} satisfies $N^-(u)=0.$ For vertices $u$ and $v$ in a digraph $G$, we write $d_G(u,v)$ for the distance from $u$ to $v$, if one exists. If there is no such directed path, then $d_G(u,v) =\infty.$ We drop the subscript $G$ if it is clear from context. A digraph is \emph{strong} (or \emph{strongly connected}) if there is a directed path connecting every pair of vertices. The \emph{strong components} (or \emph{strongly connected components}) of a digraph are its maximal strongly connected subdigraphs. A \emph{tournament} is a digraph where for each pair of distinct vertices there is exactly one arc. For a digraph $G$ and $X\subseteq V(G)$, we let $G[X]$ be the subdigraph of $G$ induced by $X$. For a positive integer $n,$ let $[n] = \{1,2,\dots ,n\}.$ We use $\log x$ to denote the natural logarithm of $x$. For background on digraphs, the reader is directed to \cite{bang-jensen-gutin-2009} and to \cite{west} for background on graphs.

\section{Bounds on the localization number}

Finding exact values for the localization number of digraphs may be challenging. We focus instead on finding bounds on the localization number in terms of their order, strong components, and various width parameters.

\subsection{Extremal values}

A natural question is consider digraphs with small and large localization numbers. This thread of research is motivated by the classification of cop-win graphs \cite{nw,q,q2}, and by Meyniel's conjecture on the largest asymptotic value of the cop number \cite{bonato1,bonato2}.

In the case when distance probes are only considered to be finite, the digraphs with metric dimension $1$ were classified in  \cite{chartrand-raines-zhang-2000}. We extend this characterization to the more general setting allowing infinite distance probe.
\begin{theorem}  \label{lem:MetricDim1}
A digraph $G$ on $n$ vertices has metric dimension $1$ if and only if the following properties hold.
\begin{enumerate}
\item The digraph $G$ has a directed Hamiltonian path $(v_1, v_2,\ldots, v_n)$ with no arc connecting $v_i$ to $v_j$ if $i < j-1$; or
\item The digraph $G$ has a source vertex, and removing this source vertex yields a digraph that has a directed Hamiltonian path $(v_1, v_2,\ldots, v_{n-1})$ with no arc connecting $v_i$ to $v_j$ if $i < j-1$.
\end{enumerate}
\end{theorem}
\begin{proof}
Let $G$ have vertices $v_1, v_2,\ldots, v_n$.
Suppose that a cop is on vertex $v_1$ and is able to capture the robber in one round.
Each distance $d(v_1,v)$ must be unique for every vertex $v$ in the graph.
We then have that either in the first case that $\{d(v_1, v_2), d(v_1,v_3), \ldots, d(v_1, v_n)\}$ is equal to $\{1,2,\ldots, n-2, n-1\}$  or $\{1,2,\ldots, n-3, n-2, \infty\}$ in the second case. We may assume without loss of generality that $d(v_1,v_i)=i-1$ for $i\in[n-1]$, and that $d(v_1,v_n)=n-1$ in the first case and $d(v_1,v_n)=\infty$ in the second case.

Let $P_i$ be the directed path of length $d(v_1,v_i)$ from $v_1$ to $v_i$. The second last vertex of this directed graph must have distance $d(v_1,v_i)-1 = d(v_1,v_{i-1})$ to the cop, and as each vertex has a unique distance to the cop, this vertex must be $v_{i-1}$.
If there was an arc from $v_i$ to $v_j$ with $i<j-1$, then there would be a path of length $$d(v_1,v_i)+1 = (i-1)+1 = i<j-1 =d(v_1,v_j)$$ from $v_1$ to $v_j$, which is a contradiction.

In the first case, the Hamiltonian path $(v_1, \ldots, v_n)$ satisfies the required conditions.
In the second case, the vertex $v_n$ is a source vertex, and after removing this, $(v_1, v_2,\ldots, v_{n-1})$ is the Hamiltonian path.
\end{proof}

Theorem~\ref{lem:MetricDim1} provides a class of digraphs with localization number $1$, but classifying all those digraphs with localization number $1$ remains open. We will show that the localization number of an acyclic digraph is 1.

For a digraph $G,$ an ordering of its vertices $x_1, x_2, \ldots, x_n$ is a \emph{topological sort} if for every arc $(x_i,x_j)$ in $G$ we have that $i < j$.  A digraph $G$ \emph{acyclic} if it has no cycles. Note that every acyclic digraph has a sink and a source and has a topological sort of its vertices; see, for example,~\cite{bang-jensen-gutin-2009}.

\begin{theorem} \label{thm:DAG1}
If $G$ be an acyclic digraph, then $\zeta(G) = 1$.
\end{theorem}
\begin{proof}
Let $x_1, x_2, \ldots, x_n$ be a topological sort of $V(G)$, with $|V(G)|=n$. The strategy of the cop is to move iteratively to higher indexed vertices. In round $i$, where $i\ge 0$ is an integer, the cop probes vertex $x_i$.

If the robber resides on vertex $x_i$, then the robber must move to a vertex with a higher index before round $i$ to avoid capture. In this way, the cop can force the robber to move to a sink in at most $n$ rounds, where the robber will be captured.
\end{proof}

We next investigate two classes of digraphs whose ratio of localization number and order is high. The \emph{rotation-symmetric tournament} $T_{2\oddtreevar+1}$ has vertices in $[2\oddtreevar+1]$, with an edge from vertex $i$ to vertex $i+j$ for $j \in [\oddtreevar]$, where $i+j$ is taken modulo $2\oddtreevar+1$ (where we identify $2m+1$ with $0$).
Note that for any fixed $c$, the function mapping vertex $i$ to vertex $i+c$ for all $i \in [2\oddtreevar+1]$ is an automorphism, and so the rotation-symmetric tournament is vertex-transitive.

\begin{lemma}\label{lem:rotation-symmetric-tournaments}
For the rotation-symmetric tournaments $T_{2m+1}$, we have that
$$\zeta(T_{2m+1}) = \lfloor \oddtreevar/2 \rfloor+1  = (1+o(1))n/4.$$
\end{lemma}
\begin{proof}
The transitive tournament of order $3$ covers the case $m=1,$ so assume $\oddtreevar \geq 2$.
We show that the rotation-symmetric tournament $T_{2\oddtreevar+1}$ has localization number $\lfloor \oddtreevar/2 \rfloor+1$. We begin by deriving an upper bound by playing with $\lfloor \oddtreevar/2 \rfloor+1$ cops, and showing that the robber can be captured.

At the start of the game, after the robber chooses their vertex, we place the cops on vertices $4s$ for $0 \leq s \leq \lfloor \oddtreevar/2 \rfloor$.
The cop on vertex $4s$ will probe $1$, and the cop on vertex $4(s+1)$ will probe $2$ if and only if the robber is on one of the vertices $\{4s+1, 4s+2, 4s+3\}$.
Similarly, the cop on vertex $4\lfloor \oddtreevar/2 \rfloor$ probes $1$ and the cop on vertex $0$ probes $2$ if and only if the robber is on one of the vertices $\{v : v\geq 4\lfloor \oddtreevar/2 \rfloor\} \subseteq \{2\oddtreevar-1, 2\oddtreevar, 2\oddtreevar+1\}$. In either case, the cop knows that the robber is on one of three consecutive vertices. Otherwise, the robber is on a vertex of the form $4s$ and is captured.

Suppose the cop found the robber was on one of the vertices $\{R, R+1, R+2\}$. The robber moves, and so is now on one of the vertices in $\{R, R+1, \ldots, R+\oddtreevar+2\}$. The cops make their second move by placing a cop on vertices $R+2s+1$ for each $s \in \{0,1, \ldots, \lfloor \oddtreevar/2 \rfloor\}$. The cop on vertex $R+2s+1$ will probe $1$, and the cop on vertex $R+2(s+1)+1$ will probe $2$ if and only if the robber is on the vertex $R+2s+2$. The cop on vertex $R+2s+1$ will probe $0$ if and only if the robber is on vertex $R+2s$.
If $\oddtreevar$ is odd, then all these cops probe a distance of $2$ if and only if the robber is on vertex $R$.
If $\oddtreevar$ is even, then the cop on $R+\oddtreevar+1$ probes $1$ while all other cops probe a distance of $2$ if and only if the robber is on vertex $R$. Note also that in this case, the cop on vertex $R+\oddtreevar+1$ probes $0$ if and only if the robber is on $R+\oddtreevar+1$.

Therefore, the robber is captured if $\oddtreevar$ is even. If $\oddtreevar$ is odd, then the robber is captured if it is on $\{R,R+1, \ldots, R+1+2\lfloor \oddtreevar/2\rfloor\} = \{R,R+1, \ldots, R+\oddtreevar\}$. Assume that the robber is not captured, and so $\oddtreevar$ is odd and the robber is residing on one of the two consecutive vertices, $R'=R+\oddtreevar+1$ or $R'+1=R+\oddtreevar+2$.

The robber moves and is now on one of the vertices in $\{R',R'+1, \ldots, R'+\oddtreevar+1\}$. The cops make their third move by placing a cop on vertices $R'+2s+1$ for each $s \in \{0, 1,\ldots, \lfloor \oddtreevar/2 \rfloor\}$. As in the previous moves, the robber is captured if it is on $\{R',R'+1, \ldots, R'+1+2\lfloor \oddtreevar/2\rfloor\} = \{R', R'+1, \ldots, R'+\oddtreevar\}$. All cops probe a distance of $1$ if and only if the robber is on $R' + \oddtreevar + 1$. Therefore, the robber is captured on any of its moves. This proves the upper bound, as we have shown that $\lfloor \oddtreevar/2 \rfloor +1$ cops is sufficient to capture the robber.

We proceed to show that $\lfloor \oddtreevar/2 \rfloor +1$ cops are necessary to capture the robber. Suppose for the sake of contradiction that $\lfloor \oddtreevar/2 \rfloor$ cops are enough to capture the robber. Suppose that during some cop move, the cops cannot distinguish whether the cop is on vertex $R$ or $R+1$.
The robber moves to some vertex in $\{R, R+1,\ldots, R+\oddtreevar+1\}$.

There are two types of placements that the cops can make to distinguish these vertices.
The first type are those of the form $R -\oddtreevar+j$ for $1 \leq j \leq \oddtreevar-1$, which are those vertices that the robber could not have moved to from vertex $R$ or $R+1$. If a cop plays on vertex $R-\oddtreevar+j$ for $1 \leq j \leq \oddtreevar-1$, then it will probe a distance $1$ to all vertices in $A_j=\{R,R+1, \ldots, R+j\}$ and a distance $2$ to all vertices in $B_j=\{R+j+1, R+j+2,\ldots, R+\oddtreevar+1\}$, and so gives the ability to distinguish vertices in $A_j$ from those in $B_j$.
The second type of placements are those vertices that the robber could have moved to from vertex $R$ or $R+1$.
If a cop plays on vertex $R+j$ for $1 \leq j \leq \oddtreevar$, then it will probe a distance $0$ to $R+j$, a distance $1$ to all vertices in $\{R+j+1, R+j+2,\ldots, R+\oddtreevar+1\}$, and a distance $2$ to all vertices in
$\{R, R+1,\ldots, R+j-1\}$, and so gives the ability to distinguish vertices in $A_j$ from those in $B_j$, but also to distinguish vertex $R+j$ from those in $A_j\setminus \{R+j\}$.
Thus, the move on $R+j$ is strictly superior to the move on $R+j-\oddtreevar$, so we may assume the cop does not play on vertices in $\{R-\oddtreevar+1, R-m+2,\ldots, R-1\}$ in their strategy.

Therefore, there must be an optimal cop strategy where the cops only play on $\{R, R+1,\ldots, R+\oddtreevar+1\}$.
Under a cop strategy where the cops only play on $\{R, R+1,\ldots, R+\oddtreevar+1\}$, if there are two consecutive vertices $R+j$ and $R+j+1$ with $1 \leq j \leq \oddtreevar-1$ that do not contain a cop, then these two vertices cannot be distinguished by the cops.
When $\oddtreevar$ is odd, the only way of doing this with $\lceil \oddtreevar/2\rceil$ or fewer cops is by placing $\lceil \oddtreevar/2\rceil$ cops on $\{R + 2 + 2s : 0 \leq s \leq \lfloor \oddtreevar/2 \rfloor-1\}$.
However, now all cops have distance $2$ to both $R$ and $R+1$, and so these two consecutive vertices cannot be distinguished.
For the case that $\oddtreevar$ is even, there are two ways to ensure that any two consecutive vertices in $\{R+1, R+2,\ldots, R+\oddtreevar+1\}$ contains a cop,
by placing $\lceil \oddtreevar/2\rceil$ cops on $\{R + 1 + 2s : 0 \leq s \leq \lfloor \oddtreevar/2 \rfloor-1\}$, or by placing $\lceil \oddtreevar/2\rceil$ cops on $\{R + 2 + 2s : 0 \leq s \leq \lfloor \oddtreevar/2 \rfloor-1\}$.
In the first case, all cops have distance $1$ to the consecutive vertices $R+\oddtreevar$ and $R+\oddtreevar+1$, and in the second case, all cops have distance $2$ to the consecutive vertices $R$ and $R+1$.
In all cases, we found two consecutive vertices that the cops cannot distinguish. Therefore the robber may, in each round, move to one of the two consecutive vertices that the cop will not be able to distinguish on their next move. The robber can avoid capture indefinitely using this strategy, which completes the contradiction.
\end{proof}

The metric dimension of strong tournaments of order $n$ was shown in \cite{loz} to be bounded above by $\lfloor n/2 \rfloor$.
We give general upper bounds on the localization number of tournaments.

\begin{theorem}
Let $T$ be a tournament and let $T'$ be the maximal strong subtournament of $T$ with maximum localization number. We then have that
\[ \zeta(T) \leq \zeta(T')+1.
\]
In particular,
\[ \zeta(T) \leq \lfloor n/2 \rfloor,
\]
except when $T$ can be formed by adding a sink to a strong, even order tournament $T'$ that has $\zeta(T')=|T'|/2$, and so $\zeta(T) \in  \{|T'|/2,|T'|/2+1\} = \{(n-1)/2, (n+1)/2\}$.
\end{theorem}
\begin{proof}
Let $T_1, \ldots, T_m$ be the set of maximal strong components of $T$, with indices chosen so that if $u \in T_i$ and $v \in T_j$ for $i<j$ then there is a directed arc from $u$ to $v$.
Starting at $i=1$ and increasing $i$ until $i=m-1$, on round $i$ the cops place $\beta(T_i)$ cops on a metric basis of $T_i$ and one cop on $T_i+1$.
In round $m$, the cops place $\beta(T_m)$ on a metric basis of cops on $T_m$.

On round $i$, the $\beta(T_i)$ cops will find that the robber is either on a single vertex of $T_i$ or that the robber is on some $T_j$ with $j>i$. If $i\neq m$, then the single cop distinguishes whether the robber is on $T_i$ (as it probes infinity) or is not on $T_i$ (as it probes a finite value). As such, the robber is either captured on $T_i$, or the cops know the robber is on $T_j$ for $j>i$.
Note that when $i=m$, the $\beta(T_m)$ cops will capture the robber.

The method in the previous paragraph requires $\max_i\{\beta(T_i)\}+1 \leq \max_i\{\lfloor|V(T_i)|/2\rfloor\}+1$ cops when $m>1$. This is at most $\lfloor n/2 \rfloor$ when the exceptional case does not occur. The case when $m=1$ follows immediately since the metric dimension of a strong tournament is $\lfloor n/2 \rfloor$.
\end{proof}

In the introduction, we discussed the digraph $D_3(i)$ whose localization number was 1/3 of its order.
Using the tournaments in Lemma~\ref{lem:rotation-symmetric-tournaments}, we may find digraphs of order $n$ with localization number asymptotically 1/2 of their order.

\begin{theorem}\label{teg}
If $n=m^2$ with $m$ an odd integer, then
there is a directed graph $G$ of order $n$ such that
$$\zeta(G)=n/2-\sqrt{n}+3/2 = (1-o(1))n/2.$$
\end{theorem}
\begin{proof}
Consider the rotation-symmetric tournament $T_{2j+1}$ and replace each vertex $v$ by an independent set $I_v$ of size $k\geq 3$, with an edge from $x\in I_u$ to $y \in I_v$ exactly when there is an edge from $u$ to $v$ in the rotation-symmetric tournament. Label the resulting digraph as $G$.
We will show that $(k-1)j+1$ cops are both necessary and sufficient to capture the robber on $G$.

We begin by showing sufficiency, and play with $(k-1)j+1$ cops.
For the first cop move, the cop can place one cop on each independent set, using $2j+1\leq (k-1)j+1$ cops. The unique cop that is in the same independent set as the robber will probe a distance of $0$ or $3$, and all other cops will probe a distance of $1$ or $2$. Therefore, the cops know the independent set that the robber is in, say independent set $I_1$ without loss of generality.
The robber may move to any vertex in $\bigcup_{2 \leq i  \leq j+1} I_i$, or stay on its current vertex.
The cops may then move $k-1$ cops on each independent set $I_i$ with $i \in \{2, \ldots, j+1\}$. The cops also move an additional cop on an unvisited vertex of $I_1$. If the robber moves to a vertex in $\bigcup_{2 \leq i  \leq j+1} I_i$, then some cop will either probe $0$ (so the robber is captured), or all cops on the same independent set as the robber will probe $3$. In the latter case, the robber is captured on the unique vertex of this independent set that does not contain a cop.
Therefore, the robber is captured unless it stays on $I_1$. The cops repeat this technique, with the one cop on $I_1$ scanning over each vertex in $I_1$ until all have been visited, at which point the robber must be captured.

To show necessity, we play with just $(k-1)j$ cops, and show that the robber can avoid capture indefinitely. Suppose for the sake of contradiction that after the cops' second last move, the cops knew the robber was on some independent set, say $I_1$, but could not distinguish which of two vertices $u,u' \in I_1$ that the robber was on, and that the cops can capture the robber on the next cop move. Note that $G$ is vertex-transitive, so the following argument works on any robber location. The robber may move to any vertex in $\bigcup_{2 \leq i  \leq j+1} I_i$.
If any $I_i$ for $2 \leq i \leq j+1$ contain fewer than $k-1$ cops on the next cop move, then there will be two vertices in $I_i$ that the robber can occupy but that the cops cannot distinguish. Therefore, there must be $k-1$ cops on these $j$ independent sets.
There are no further cops to be utilised, and since no cop outside of $I_1$ can distinguish $u,u' \in I_1$, there are two vertices in $I_1$ that the cops cannot distinguish, and so the robber is not captured.

Thus, $(k-1)j+1$ cops are required to capture the robber, and $(k-1)j$ cops are insufficient to do so.
With $n=k(2j+1)$, and by choosing the optimal values of $k=\sqrt{n}$ and $2j+1=\sqrt{n}$, the result follows.
\end{proof}

In undirected graphs, we have that for complete graphs $\zeta(K_n) = \beta(K_n) = n-1$. For digraphs, however, there is no obvious example of such a large localization number (recall that we take our directed graphs to be oriented). A result of \cite{chartrand-raines-zhang-2000} showed that there is a family of digraphs of order $n$  with metric dimension $n-3$. However, these graphs have localization number at most $2$. We showed that there are families of digraphs with $\zeta(G_n) = (1-o(1))n/2$, but there are no obvious analogues to the examples given in the other cases. We conjecture that there exists a constant $1/2 \le \varepsilon < 1$ such that for digraphs $G$ on $n$ vertices, $\zeta(G) \le \varepsilon n.$

\subsection{Upper bounds using strong components}

For a digraph $G$, we may contract the strong components to single vertices and delete any resulting parallel edges to create the \textit{strong component digraph} $\mathrm{SC}(G)$. The strong components digraph of $G$ is acyclic and thus, has a topological sort. We say a strong component of $G$ is $\emph{terminal}$ if its corresponding vertex in $\mathrm{SC}(G)$ has out-degree $0$. We say a strong component of $G$ is $\emph{source}$ if its corresponding vertex in $\mathrm{SC}(G)$ has in-degree $0$. A strong component $G_j$ is a \emph{child} of the strong component $G_i$ if there exists an arc from a vertex in $G_i$ to a vertex in $G_j$, and we say $G_j$ is a \emph{descendant} strong component of $G_i$ in $G$ if there existed a directed path from the vertex of $G_i$ to the vertex of $G_j$ in $\mathrm{SC}(G)$. Label the vertices in $G_i$ and all descendent strong components of $G_i$ as $\text{Desc}[G_i]$. If $G$ is a digraph, then $\scmax^+(G)$ is the maximum out-degree of $G.$

\begin{theorem} \label{thm:ub-sc} 
Let $G_1, G_2, \ldots G_k$ be the strong components of $G$ and $\mathrm{SC}(G)$ be the strong component digraph of $G$. If $\scmax = \Delta^+(\mathrm{SC}(G))$, then $$\zeta(G) \le \max_{1\le i \le k} \zeta(G_i) + \scmax.$$
\end{theorem}
\begin{proof}
We follow the cop strategy of Theorem~\ref{thm:DAG1} on $SC(G)$. Let $x_1, x_2,\ldots, x_m$ be a topological sort of the vertices of $SC(G)$, where each $x_i \in SC(G)$ corresponds to a strong component $G_i$ in $G$. In particular, if there is a directed edge between a vertex of $G_i$ and a vertex of $G_j$, then $i < j$.
The cops play in $m$ distinct phases, focusing on the strong component $G_i$ in phase $i$.
In phase $i$, we show that the cops play to discover if the robber is on $G_i$, and then in addition to make the robber leave $G_i$ (or be captured) if it is on $G_i$.
The robber can only move to a strong component with a higher index than it is currently on, and so at the start of phase $i$, the robber must be on a strong component with index $i$ or larger.
In this way, the cop can capture the robber by the end of phase $m$.

Now suppose we enter phase $i$.
We will play $\zeta(G_i)$ cops on $G_i$ and at most $\scmax$ cops outside of $G_i$.
The $\zeta(G_i)$ cops on $G_i$ play by ignoring candidate vertices outside of $G_i$.
It is important to note here that since $G_i$ is a maximal strongly connected component of $G$, no path in $G$ from $u$ to $v$ with $u,v \in V(G_i)$ contains a vertex outside of $V(G_i)$, and so $d_{G_i}(u,v) = d_G(u,v)$.
That is, playing the Localization game on the digraph $G_i$ (not taking $G_i$ as being a subgraph of $G$, but being its own digraph) is identical to playing the Localization game on $G$, where both cops and robber restrict themselves to playing on the vertices $V(G_i)$ in $G$.

In addition, for all rounds in this phase, one cop is placed on a vertex in each child strong component of $G_i$.
Note that if a cop on a child component $G_j$ of $G_i$ probes a finite distance, then the robber is not on $G_i$, since there is no path from a vertex of $G_j$ to a vertex of $G_i$.
As such, the robber is known to not reside on $G_i$, and phase $i$ immediately ends.
As such, if the robber ever moves off from $G_i$ (or was never on $G_i$), then phase $i$ ends on the next cop move. Otherwise, we may assume the robber restricts itself to playing on the vertices $V(G_i)$ in $G$, and so is captured by the $\zeta(G_i)$ cops on $G_i$.
\end{proof}

The bound in Theorem~\ref{thm:ub-sc} is tight, although in general it can be arbitrarily far from the truth. For tightness, we provide a class of digraphs that meet the bound.

\begin{theorem}\label{thm:tight-sc} 
For each odd $\oddtreevar$ and $\scmax k\leq \frac{\oddtreevar+1}{2}$, there is a digraph $G$ of order $(2 \oddtreevar+1) \scmax$ vertices with strong components $G_1, G_2, \ldots G_{\scmax+1}$ and $\scmax = \Delta^+(\mathrm{SC}(G))$, such that \[\zeta(G) = \max_{1\le i \le \scmax +1} \zeta(G_i) + \scmax.
\]
\end{theorem}
\begin{proof}
Consider the graph on vertices $\mathbb{Z}_{2\oddtreevar+1} \times [\scmax+1]$, where an arc connects $(u_1,v_1)$ to $(u_1+j,v_2)$ for each $1 \leq j \leq \oddtreevar$ whenever $v_1 = 1$ or $v_1 = v_2$ (with $u_1 + j$ taken modulo $2\oddtreevar+1$).
That is, each set of vertices $V_i=\{(u,i) : u \in \mathbb{Z}_{2\oddtreevar+1}\}$ induces a copy of $T_{2\oddtreevar+1}$ in $G$.
Further, $G[V_i] = G_i$ is a strong component of $G$, $N^+((u,1)) = N^+_{G_1}(u) \times [\scmax+1]$, and $N^+((u,j)) = N^+_{G_1}(u) \times \{j\}$.

For this digraph, $\zeta(G_1) = \zeta(G_i)$ for all $1 \leq i \leq \scmax+1$. As we have the upper bound from Theorem \ref{thm:ub-sc}, we must show that $$\zeta(G) \geq \max_{1\le i \le k} \zeta(G_i) + \scmax = \zeta(G_1) + \scmax.$$
Recalling that we have assumed $\oddtreevar$ is odd,  Lemma~\ref{lem:rotation-symmetric-tournaments} yields that $\zeta(G_1)=\frac{\oddtreevar+1}{2}$.
Therefore, for the sake of contradiction, assume that $\frac{\oddtreevar+1}{2} + \scmax-1$ cops play on $G$, and are able to capture the robber.
We will analyze the last move that the cops made in order that the robber was either captured or was forced to move off from $G_1$ in order to avoid capture.

Suppose that the robber was on vertex $(1,1)$, which can be assumed without loss of generality since $T_{2\oddtreevar+1}$ is vertex-transitive.
By a straightforward modification to the argument found in Lemma \ref{lem:rotation-symmetric-tournaments}, we can show that for each vertex $u$ in $G_1$ that is not in $N^+_{G_1}[(1,1)]$ is a strictly inferior move for a cop compared to playing on the vertex $u+\oddtreevar$ or $u+\oddtreevar+1$ in $N^+_{G_1}[(1,1)]$.
As such, we may assume the cop only plays on vertices in $N^+_{G_1}[(1,1)]$ when playing on $G_1$.
As we are playing with $\frac{\oddtreevar+1}{2} + \scmax-1 \leq \frac{\oddtreevar+1}{2}+\frac{\oddtreevar+1}{2} - 1 = \oddtreevar$ cops, we cannot place a cop on each vertex $N^+_{G_1}[(1,1)]$, since there are $\oddtreevar+1$ vertices in this set. Thus, there must be some vertex $(w,1) \in N_{G_1}[(1,1)]$ that did not contain a cop during this cop move.
Note that if the cops can distinguish $(w,1)$ from $(w,j)$, then it must be that there is a cop on the strong component $G_j$.
To see this, note that any cop on strong component $G_i$ with $i \notin \{1,j\}$ will probe an infinite distance to any vertex in $G_j$ and that the distance from a cop in $G_1$ to $(w,1)$ and $(w,j)$ is always the same.

Therefore, each of the $\scmax$ strong components $G_j$, with $2 \leq j \leq \scmax+1$, must contain a cop.
This means that at most $(\frac{\oddtreevar+1}{2} + \scmax-1) - \scmax = \frac{\oddtreevar-1}{2}$ cops were on $G_1$ during the cops' last move.
However, on $G_1$, the set $N^+_{G_1}[(1,1)]$ contains $\oddtreevar+1$ consecutive vertices in $G_1$. There is no way to place $\frac{\oddtreevar-1}{2}$ cops on $\oddtreevar+1$ vertices so that any two adjacent vertices contain a cop. In particular, note the best that we could do is either to place the cops on $\{(1+2b,1) : 0 \leq b \leq \frac{\oddtreevar-1}{2}\}$ or to place the cops on $\{(2+2b,1) : 0 \leq b \leq \frac{\oddtreevar-1}{2}\}$.
Hence, $G_1$ contained two adjacent vertices that did not contain a cop in the cops' last move. These two vertices cannot be distinguished by any of the cops on $G$, which is a contradiction.
\end{proof}

Theorems~\ref{thm:ub-sc} and \ref{thm:tight-sc} show that both the internal structure of the strong components of a digraph along with how those strong components are connected may increase the localization number of the digraph.
The structure of how a strong component is connected to the rest of the digraph may also decrease the localization number. This phenomenon was mentioned by \cite{chartrand-raines-zhang-2000}, although we give a more comprehensive example.

\begin{lemma}
For any digraph $D$ on $m$ vertices, there exists a digraph $G$ that contains $D$ such that $\beta(G) \leq \lceil \log_2(m)\rceil$.
\end{lemma}
\begin{proof}
Label the vertices in $D$ as the binary string equivalent of $0$ to $m-1$.
Let $D'$ be the graph formed from $D$ by adding $\lceil \log_2(m)\rceil$ vertices $V'=\{v_i\}$ such that vertex $v_i$ has an arc to those vertices in $D$ whose binary string label has a $0$ in the $i$th coordinate.

To capture the robber, place a cop on each vertex in $V'$. If the robber is on a vertex in $V'$, then a cop probes a distance of $0$, and they are captured.
So suppose all the cops probe a distance of $1$ or more.
If the robber is on a vertex with binary string $b_0b_1\ldots b_{\lceil \log_2(m)\rceil}$, then the cop on $v_i$ will probe a distance of $1$ exactly when $b_i=0$. Otherwise, the cops can deduce that $b_i=1$, and so the cops know the entire binary string of the robber, and the robber is captured.
\end{proof}

\subsection{Upper bounds using hypergraphs}

Theorem~\ref{thm:ub-sc} is especially useful in bounding the localization number of digraphs with many strong components. In the case that the digraph is dense (or even a tournament), it may have few strong components and so Theorem \ref{thm:ub-sc} provides no information.
For those dense digraphs, there is another method that can help us to understand an upper bound on their metric dimension, and hence their localization number as well.

We give the following approach using linear programming on hypergraphs, that may be of interest in its own right. A \emph{hypergraph} is a discrete structure with vertices and \emph{hyperedges}, which consists
of sets of vertices. Graphs are special cases of hypergraphs, where each hyperedge has cardinality two. For a hypergraph $\mathcal{H}$, a \emph{vertex cover} is a subset of vertices $V' \subseteq V(\mathcal{H})$ such that each hyperedge of $\mathcal{H}$ contains at least one vertex in $V'$. The minimum size of a vertex cover in $\mathcal{H}$ is denoted $\tau(\mathcal{H})$.

There is an equivalent way to define $\tau$.  Let $V(\mathcal{H}) = \{v_1, \ldots, v_n\}$.
For a subset of vertices $V'$, we define $x_i = 1$ if $v_i\in V'$ and $x_i = 0$, otherwise. In this case, $\sum_{v_i \in h} x_i\geq 1$ for all $h\in E(\mathcal{H})$ if and only if $V'$ is a vertex cover.
As such, the following integer program yields $\tau(\mathcal{H})$:
\[
\text{Minimize } \sum{x_i} \text{ subject to } \sum_{v_i \in h} x_i\geq 1 \text{ for } h \in \mathcal{H} \text{ and } x_i \in \{0,1\}.
\]

The \emph{fractional vertex-cover number} of $H$, written $\tau^*(\mathcal{H})$, is the resulting value of the linear program that drops the requirement that $x_i$ is an integer in $\{0,1\}$, and ensures that $0 \leq x_i \leq 1$ instead. The following bound is useful.

\begin{theorem}[\cite{lovasz1}]\label{thm:lovasz1}
Let $\mathcal{H}$ be a hypergraph where each vertex is contained in at most $d$ hyperedges. We then have that
\[
\tau(\mathcal{H}) \leq (1+\log d) \tau^*(\mathcal{H}).
\]
\end{theorem}

We now provide a new bound on the metric dimension.

\begin{theorem}\label{th}
Let $G$ be a digraph such that for any pair of vertices $x,y$ in $G$, at least $cn$ vertices of $G$ have a different distance to $x$ and to $y$. We then have that
\[\beta(G) \leq \frac{1+2\log n}{c}.\]
\end{theorem}
\begin{proof}
Consider the hypergraph $\mathcal{H}$ defined from $G$ as the hypergraph with vertices $v_1, v_2\ldots, v_n$ and hyperedges $h_{ij} = \{x\in V(G) : d(v_i,x) \neq d(v_j,x)\}$ for each $v_i,v_j \in V(G)$ with $i<j$. Observe that there are $\binom{n}{2}$ hyperedges in this hypergraph.

We will find the size of a vertex cover of this hypergraph, $\tau(\mathcal{H})$.
Set $x_{i} = 1/cn$ for each $v_i \in V(G)$.
For each $h\in \mathcal{H}$, we have that $$\sum_{v_{i} \in h} x_i = \sum_{v_{i} \in h} \frac{1}{cn} \geq cn \frac{1}{cn}=1,$$ and so the conditions of the linear program are satisfied,  yielding $\tau^*(\mathcal{H}) \leq \sum{x_i} = \frac{1}{c}$.
Each vertex can be in at most $d= \binom{n}{2}$ hyperedges (which is the total number of hyperedges).
Theorem~\ref{thm:lovasz1} then yields that $\tau(\mathcal{H}) \leq (1+\log (n^2)) \frac{1}{c}$.
There then exists a vertex cover on $\mathcal{H}$, say $S$, with at most $(1+\log(n^2)) \frac{1}{c}$ vertices.
Each hyperedge $h_{ij}$ contains at least one vertex in $x\in S$, which by definition yields that $d(v_i,x) \neq d(v_j,x)$ and so $x$ can distinguish $v_j$ and $v_j$ in $G$.
As this holds for all pairs of vertices in $G$, the set $S$ can distinguish all pairs, and hence, the metric dimension is exactly $\tau$, and the theorem follows.
\end{proof}

Note that there is no requirement for $c$ to be constant, so we may take $G$ with parameter $c = \frac{4\log n}{n}$ to yield $
\beta(G) \leq n/2$. Thus, if a digraph $G$ has $\zeta(G)>n/2$, then for most vertices $x$, there must be some $y$ such that very few vertices can distinguish $x$ and $y$.

\subsection{A lower bound}
We finish this section by providing a lower bound on the localization number, extending results of \cite{BK} on graphs. We denote the closed out-neighbor and in-neighbor sets of a vertex $v$ by $N^+[v]$ and $N^-[v]$, respectively. Let $G$ be a digraph. Define
$$M(G) = \max_{u,v\in V(G)} \left( \max_{w \in N^+[v]} d(u,w) - \min_{w \in N^+[v]} d(u,w) \right) + 1.$$

The parameter $M(G)$ measures the spread of values of the distances in a closed out-neighbor set.
Note that in an undirected graph, if we replace $N^+[v]$ by $N[v]$ in the definition, then $M(G) = 3.$ To see this, if $d(u,v) = d,$ then any vertex $w \in N[v]$ will have distance one of $d,d-1,d+1$ to $u.$

Define the \emph{out-degeneracy} of $G$ to be the maximum, over all subgraphs $H$ of $G$, of $\delta^+(H).$

\begin{theorem}\label{dt}
If $G$ has out-degeneracy $k,$ then $\zeta(G) \ge \log_{M(G)}(k+1).$
\end{theorem}

\begin{proof}
Assume that there are $m$ cops at play, with $m < \log_{M(G)}(k+1).$ We show that the robber has a winning strategy.

Let $H$ be a subgraph that witnesses the out-degeneracy of $G,$ so $\delta^+(H)=k$ is minimum. The robber $R$ will remain in $H$ throughout the game. Let the cops be labeled $C_i,$ where $1\le i \le m.$ Suppose that $R$ is on $v$ and the cop $C_i$ is on $u_i$ in $G$, where $1\le i \le m.$
In the next round, the robber may remain on $v$ or move to an out-neighbor $w$ of $v.$ The distances $d_G(u_i,w)$ are one of $M(G)$-many values. Hence, there are at most $M(G)^m$ possible distance vectors in that round.

By hypothesis, we have that $M(G)^m < k+1.$ Note that $k+1$ is the cardinality of ${N_H}^+[v].$ By the pigeonhole principle, there are at least two vertices in ${N_H}^+[v]$ with the same distances vectors.

Suppose that on some robber's move, the robber occupies $v \in V(H).$ If the robber is choosing an initial position, then the robber instead pretends that they already occupy some arbitrary vertex
$v$ of $H$ and wishes to move to some neighbor of $v.$ Before making
their move, the robber considers the cops' subsequent probe.  The robber finds two vertices of $H$, say $w_1$ and $w_2$ that share the same distance vector with respect to the given probe. The robber moves to $w_1$, say, and is not identified by the cops. The robber may repeat this strategy
indefinitely to avoid capture. The proof follows. \end{proof}

\section{Width parameters}
Bounds for the localization number of a graph based on width-type parameters, such as pathwidth and treewidth, were given in \cite{BBHMP,Bosek2018}. In this section, we consider directed analogues of pathwidth and treewidth, and present bounds on the localization number of a digraph using these parameters. We begin by defining the directed pathwidth of a digraph.

A directed path-decomposition of a digraph $G$ is a sequence of vertex subsets $W_1,$ $W_2,$ $\ldots,$ $W_k$ such that
\begin{enumerate}[label=(\roman*)]
    \item $\bigcup_{i=1}^k W_i=V(G)$;
    \item if $i<j<k$ then $W_i \cap W_k \subseteq W_j$; and
    \item if $(v,u)\in E(G)$, then either $u\in W_i$ and $v\in W_i$ for some $W_i$, or $u\in W_j$ and $v\in W_i$ where $i<j$.
\end{enumerate}
The \emph{width} of a directed path-decomposition is the maximum cardinality of a $W_i$ minus one. The \emph{directed pathwidth} (or simply, \emph{pathwidth}) of a digraph $G$, denoted $d\mbox{-}\mathrm{pw}(G),$ is the minimum width over all directed path-decompositions of $G$.

We have the following upper bound for the localization number using pathwidth.

\begin{theorem}\label{pwt}
$\zeta (G) \le d\mbox{-}\mathrm{pw}(G)+1.$
\end{theorem}

\begin{proof}
Let $G$ be a digraph with $d\mbox{-}\mathrm{pw}(G)=m$, and $W_1, \ldots, W_m$ be a path decomposition of $G$. We assume that for each $W_i$ for $1\le i <k$, $W_i \backslash W_{i+1}$ is non empty. Place all $m$ cops in the first bag $W_1$. If $R$ is in $W_1$, then they will be captured. Hence, we can ensure that the robber is not on $W_1$ after the first round. We say that $W_1$ has been \emph{cleared}.

Now suppose that the cops have probed the vertices of $W_i$ for some $i\ge 1.$ We claim the robber can never enter the bag in later rounds. Assume to the contrary that the robber enters some vertex $u\in W_i$ after $W_i$ has been cleared. Let $(v,u)$ be the arc through which the robber enters $W_i.$ By the definition of directed path width, there is either a bag $W_j$ containing both $u$ and $v$, or bags $W_j$ and $W_k$ where $v\in W_j,$ $u\in W_k,$ and $j<k.$

In the first case, if $j\ge i,$ then $u$ is in every bag between $W_i$ and $W_j$, so $u$ is guarded by a cop at least until the cops probe the vertices of $W_j.$ That is, $u$ is guarded until $v$ is cleared, so the robber cannot move to $u$ from $v.$ If $j < i,$ then $v\in W_i$ also, so $v$ has already been cleared and hence, the robber cannot be on $v.$ In the latter case, $u$ is guarded by a cop in bag $W_k$ while vertex $v$ was already probed in $W_j.$ The robber will not be able to enter $W_i$ through vertex $v$ without first being located by a cop. Hence, we can guarantee that the robber is not in $W_i$ and that the robber cannot re-enter $W_i.$

We proceed by induction to probe the vertices in each subsequent bag, ensuring that the robber is not in a previous bag. We continue scanning the vertices of each bag until we each the endpoint, $W_m$ and locate the robber in this bag.
\end{proof}

Note that the Theorem~\ref{pwt} bound is tight for digraphs $D$ with $d\mbox{-}\mathrm{pw}(G)=0$, such as transitive tournaments and orientations of $P_n.$

We also consider bounds on the localization number of digraphs using a directed analogue of treewidth. This parameter, called DAG-width, is defined as follows. Let $G$ be a directed acyclic graph (or \emph{DAG}). The partial order $\preceq _G$ on $G$ is the reflexive, transitive closure of $E(G).$ That is, $i \preceq_G j$ if there is a directed path from vertex $i$ to vertex $j$ in $G$.

Let $G$ be a digraph and $W,V' \subseteq V(G).$ We say that $W$ \textit{guards} $V'$ if for all $(u,v)\in E(G),$ if $u \in V'$ then $v\in V' \cup W.$ A \textit{DAG-decomposition} of a digraph $G$ is a pair $(D, X)$ where $D$ is a DAG and $X = (X _d)_{d \in V(D)}$ is a family of subsets of $V(G)$ such that:
\begin{enumerate}[label=(\roman*)]
    \item $\bigcup _{d\in V(D)} X _d =V(G);$
    \item for all $d \preceq_D d' \preceq_D d'',$ $X_d \cap X _{d''} \subseteq X_{d'};$ and
    \item for all $(d, d') \in E(D),$ $X_d \cap X_{d'}$ guards $X _{\succeq d'} \backslash X_d$ where $X _{\succeq d'}=\cup_{d' \preceq_D d''} X _{d''}.$ For any source $d,$ $X_{\preceq d}$ is guarded by $\emptyset.$
\end{enumerate}
In \cite{Biedl2015} it was shown that $(\mathrm{iii})$ is equivalent to the following property, which is more intuitive.
\begin{enumerate}[label=(\roman*)']
\setcounter{enumi}{2}
    \item For any source $j$ in $D,$ any $u \in X_j$, and any edge $(u,v)$ in $G$, there exists a successor bag $X_k$ of $X_j$ with $v \in X_k.$ Also, for every arc $(i,j) \in E(D),$ any $u \in (X_j \backslash X_i),$ and any edge $(u,v) \in E(G),$ there exists a successor-bag $X_k$ of $X_j$ with $v \in X_k,$ where a \textit{successor-bag} of $X_j$ is a bag $X_k$ with $j \preceq _D k.$
\end{enumerate}

The \textit{width} of a DAG-decomposition is $\max \{|X_d|: d\in V(D)\}.$ The \emph{DAG-width} of a digraph $G$ is the minimum width over all DAG-decompositions of $G$. We finish the section with the following bound via DAG-width.

\begin{theorem}\label{dagw}
For a digraph $G$, we have that
$\zeta (G) \le \DAG (G).$
\end{theorem}

Observe that Theorem~\ref{dagw} is tight by Theorem~\ref{thm:DAG1}, as acyclic digraphs have localization number 1.

\begin{proof}[Proof of Theorem~\ref{dagw}]
Let $G$ be a digraph and $(D, X)$ be a DAG-decomposition of $G$ of minimal width. Since $D$ is a DAG, we can consider a topological sort of the vertices of $D$ given by $0, 1, \ldots m$, where $V(D)=\{0, 1, \ldots m \}.$ Let $D'$ be the DAG in which we replace each vertex $i$ of $D$ with the corresponding bag $X_i.$ Note that $X_0$ is a source in $D'$ and $X_m$ is a sink in $D'$.

Suppose we have $w$ cops where $w=\mathrm{DAG \mbox{-} width}(G).$ We begin the game by having the cops scan the vertices in $X_0,$ and successively scan the vertices of each bag in the topological sort from $X_0$ to $X_m.$ Since we have $w$ cops where $w=\max_{0\le i \le m} \{|X_i|\},$ the cops can scan every vertex in a bag in a single round.

From conditions (ii) and (iii) of the definition of DAG-decomposition, we have that if there is an edge $(u,v)\in G$, then $u\in X_i$ and $v\in X_j$ for some $i\le j,$ and for any edge $(i,j)\in D,$ any arcs between the vertices of $X_i$ and $X_j$ are directed from $X_i$ to $X_j.$ Arcs in $D'$ are directed from a lower-indexed bag to a higher-indexed bag, and any arc $(u,v) \in G$ is either contained within a bag or there is a bag $X_j \in D'$ such that $u\in X_j$ and a successor bag $X_k\in D'$ of $X_j$ such that $v\in X_k.$ Thus, once a bag $X_i$ is cleared by the cops, the robber cannot enter any previously cleared vertex (that is, a vertex that appears before $X_i$ in the topological sort).

Therefore, by successively clearing the bags from $X_0$ to $X_m$, the cops will force the robber into some bag corresponding to a sink in $D'$ as $D'$ has no directed cycles, and will capture the robber there.
\end{proof}

\section{Random and quasi-random tournaments}

In this section, we consider the localization number of both random and quasi-random tournaments.

\subsection{Random tournaments}
We first define random tournaments, which are natural analogues of binomial random graphs $G(n,p).$ Let $T(n,p)$ be a tournament on $[n]=\{1,2,\ldots, n\}$ such that the pair $(i,j)\in T(n,p)$ for $i < j$ with probability $p$ and $(j,i)\in T(n,p)$ with probability $1-p$ independently for each distinct $i,j \in [n]$.
{\L}uczak, Ruci\'nski and Gruszka studied the the behavior of this graph with regards to small subgraphs and strong connectivity for a wide range of $p$; see \cite{luczak-ruinski-gruszka-1996}. We say that an
event in a probability space holds \emph{asymptotically almost surely} (or \emph{a.a.s.}) if the probability
that it holds tends to 1 as $n$ goes to infinity. In particular, they showed that $T(n,p)$ is strongly connected a.a.s.\ provided that $np\rightarrow \infty$. This was studied in the case of $p = 1/2$~\cite{moon-moser-1962}, which is equivalent to the choosing a uniform random tournament from the set of all tournaments. For more background on random graphs and structures, the reader is directed to \cite{bp}.

We first establish some basic facts about $T(n,p)$ for constant $p$. For this, we need some notation. For a tournament $G$, let $\chi_G: V^2 \rightarrow \{-1,1\}$ be the \emph{arc indicator} of $G$: for distinct vertices $u$ and $v$, define
$$
\chi_G(u,v) = \begin{cases}1 & \textrm{if }(u,v)\in A\\-1 & \textrm{if }(u,v)\not\in A.\end{cases}
$$
Define $S(u,v)$ as the \emph{sameness set} for two vertices $u,v\in V$ by
$$
S(u,v) = \{z\in V : \chi_G(u,z) = \chi_G(v,z)\}
$$
and $s(u,v) = |S(u,v)|$, $\overline{S}(u,v) = V\setminus S(u,v)$, and $\overline{s}(u,v) = |\overline{S}(u,v)|$. We will focus on $s(u,v)$ in this subsection, and utilize $\overline{S}(u,v)$ and $\overline{s}(u,v)$ when we discuss quasi-random graphs.

We first establish some useful asymptotic properties of $T(n,p)$. For an event $A$ on a
probability space, we let $\mathrm{Pr}(A)$ denote the probability of $A.$ Given a random variable
$X,$ we let $\mathbb{E}(X)$ be the expectation of $X.$

\begin{lemma}\label{ranl}
In $T(n,p)$, if $0 < p < 1$ is constant, then a.a.s.\ we have the following.
\begin{enumerate}[label=(\roman*)]
\item For all
$x \in V$,
\begin{align*}
(1-\varepsilon)\min(1-p, p)(n-1)\le |N^+(x)| \le  (1+\varepsilon)(1-\min(1-p, p))(n-1).
\end{align*}
\item For all $x < y \in V$,
\begin{equation*}
2(1-\varepsilon)(1-p)p (n-2)\le s(x,y) \le (1+\varepsilon)((1-p)^2 + p^2) (n-2).
\end{equation*}
where $\varepsilon = 1/\sqrt {\log n}$.
\item $T(n,p)$ has diameter $2$.
\end{enumerate}
\end{lemma}
\begin{proof}
Without loss of generality we may assume that $p \le 1/2$ (that is, $p \le 1-p$) since the calculations with $p> 1/2$ are symmetric. We then have that
$$
\mathbb E [|N^+(x)|] = (1-p)(x-1) + p(n-x) \ge p(n-1).
$$
Using Chernoff's bound (see, for example, \cite{JLR}), we have that
$$
\Pr\left(||N^{+}(x)| - \mathbb E [|N^+(x)|]|\ge \varepsilon \mathbb E [|N^+(x)|]\right) \le 2 \exp\left(-\frac{\varepsilon^2 p(n-1)}{3}\right).
$$
Using the union bound gives us that a.a.s.
$$
|N^{+}(x)| = (1+o(1))((1-p)(x-1)+p(n-x)).
$$
Similarly, the expected value of $s(x,y)$ for $x<y$ is
\begin{align*}
\mathbb{E}[s(x,y)] &= ((1-p)^2 + p^2) (x-1) + 2(1-p)p(y-1-x) + (p^2 + (1-p)^2)(n-y).
\end{align*}
Note that since $2(1-p)p \le p^2 + (1-p)^2$ we have that
$$
2(1-p)p (n-2)\le \mathbb{E}[s(x,y)] \le ((1-p)^2 + p^2) (n-2).
$$
Chernoff's bound and the union bound show that for all $x < y$ we have a.a.s.\ that
$$
2(1-\varepsilon)(1-p)p (n-2)\le \mathbb{E}[s(x,y)] \le (1+\varepsilon)((1-p)^2 + p^2) (n-2).
$$

Now assume that there are vertices $x < y$ with $(y,x)\in E$ but there are no  paths of length 2 from $x$ to $y$. For ease of notation, let $G=T(n,p).$ Define $N^{+-}(x,y) = \{v: \chi_G(x,v) = 1, \chi_G(y,v)=-1\}$. The expected value of $|N^{+-}(x,y)|$ is
$$
\mathbb E[|N^{+-}(x,y)|] = p(1-p) (x-1) + p^2(y-1-x) + p(1-p)(n-y).
$$
We note that
$$
p^2 (n-2)\le \mathbb E[|N^{+-}(x,y)|] \le p(1-p) (n-2),
$$
and so Chernoff's bound with the union bound gives that
$$
\mathbb |N^{+-}(x,y)| \ge (1-\varepsilon)p^2 (n-2) = \Omega(n),
$$
which is a contradiction.
\end{proof}

With Lemma~\ref{ranl} at our disposal, we may now bound the localization number of a random tournament.

\begin{theorem}
For $0<p < 1$ constant, let $\rho = p^2 + (1-p)^2$. We then have a.a.s.\ that
$$
(1+o(1))\log_{2} n \le \zeta(T(n,p))  \le (2+o(1))\frac{\log n}{\log(1/\rho)}.
$$
In particular, for almost all tournaments $G$,
$$
(1+o(1))\log_{2} n \le \zeta(G) \le (2+o(1))\log_2 n.
$$
\end{theorem}
\begin{proof}
For the lower bound, we use a technique analogous to the bound for $k$-degeneracy in~\cite{BK}.
For the upper bound, we use an approach analogous to the one in~\cite{bollobas-mitsche-pralat-2013}, where they determined the metric dimension of $G(n,p)$ for constant $p$.

We start with the lower bound. Assume that $p \le 1/2$ (the argument for the reverse inequality is similar). Fix $0 < \alpha < p$. We show that $\zeta(T(n,p)) \ge \log_{2}(\alpha n)$.
Suppose there are $m < \log_2(\alpha n)$ cops. Since $T(n,p)$ has diameter $2$, we have that the cardinality of the set of distance vectors $(d_1, d_2, \ldots, d_m)$ for the vertices of $N^+[v]$ is at most $2^m < \alpha n$. However, $|N^+[v]| \ge (1-\varepsilon)np$, and so the number of distance vectors is strictly less than the size of the closed neighborhood of $v$. Hence, by the pigeonhole principle, there exists at least two vertices of $N^+[v]$ that have the same distance vector.
Thus, the robber may always evade capture by moving to one of the vertices in $N^+[v]$ with a duplicate distance vector associated with it.

Now we turn to the upper bound, where we a derive an upper bound on the metric dimension of $T(n,p).$ Let $W$ be a randomly chosen set of
$$
k = \frac{(2+\varepsilon)}{\log(1/\rho)} \log n = O(\log n)
$$
 vertices in $T(n,p)$.
For a random choice of $W$, let $p_W$ be the probability that $W$ does not distinguish a pair $x,y \in V$. For $W$ to not distinguish $x,y$, each $w\in W$ needs to be a member of $S(x,y)$. We then have that
$$
p_W \le \left(\frac{s(x,y)}{n}\right)\left(\frac{s(x,y)-1}{n-1}\right) \cdots \left(\frac{s(x,y)-k+1}{n-k+1}\right) \le \left(\frac{s(x,y)}{n}\right)^{k},
$$
which is at most
\begin{align*}
(1+\varepsilon)^k \rho^k
&\le \exp\left(k\log \rho+k \log (1+\varepsilon)\right) \\
&= \exp\left(\frac{(2+\varepsilon)\log n}{\log(1/\rho)}\log \rho + \frac{(2+\varepsilon)\log n}{\log(1/\rho)}\log(1+\varepsilon)\right)\\
&= \exp\left(-(2+\varepsilon) \left(1 - \frac{\log(1+\varepsilon)}{\log(1/\rho)}\right)\log n\right)\\
&= n^{-2-\varepsilon(1+o(1))}\\
&\le  n^{-2}.
\end{align*}
Thus, we obtain that the expected number of pairs $x,y$ that are not distinguished is at most
$$
\frac{n^2}{2}p_W \le \frac12,
$$
so by the probabilistic method, there exists at least one $W$ that distinguishes all pairs. Hence, we have that $\beta(T(n,p)) \le k$.
\end{proof}

\subsection{Quasi-random tournaments}

Quasi-random graphs exhibit properties of large scale random binomial graphs. Such graphs are defined by a set of equivalent conditions satisfied a.a.s.\ by $G(n,p).$ We consider quasi-random tournaments as introduced by Chung and Graham in~\cite{chung-graham-1991}. We use their notation that we recap here.

Let $G = (V,E)$ be a tournament. We let $G(n)$ be a tournament on $n$ nodes.  Define $d^\pm(v,X) = |N^\pm(v) \cap X|$ and $d^\pm(X,X') = \sum_{v \in X} d^\pm(v,X')$.
An \emph{ordering} of $G = (V,E)$ is an injective mapping $\pi:V\rightarrow [n]$.
An arc $(u,v)$ is said to be $\pi$-\emph{increasing} if $\pi(u) < \pi(v)$.
Form the undirected graph $G_\pi^+$ on $V$ by creating an edge $uv$ for each $\pi$-increasing arc $(u,v)$ of $G$.
In addition, let $N^*_G(D)$ be the number of labeled copies of the subdigraph $D$ in $G$. Let $E4C$ be a sequence $(w,x,y,z)$ with
$$
\chi_T(w,x)\chi_T(x,y)\chi_T(y,z)\chi_T(z,w) = 1.
$$

The following result is the main one on quasi-random tournaments.

\begin{theorem}[\cite{chung-graham-1991}]\label{thm:quasi1}
For a $G$ a tournament with vertices $V$, the following are equivalent.
\begin{enumerate}
\item For all tournaments $G'(s)$ of order $s$,
$$
N^*_G(G'(s)) = (1+o(1))n^s 2^{-\binom s2}.
$$
\item $N^*_G(E4C) = (1+o(1))(n^4/2)$.
\item $\sum_{u,v\in V} |s(u,v)-n/2|=o(n^3)$. \label{thm:quasi:s}
\item $\sum_{u,v\in V} ||\{w \in V|\chi_T(u,w) = 1 = \chi_T(v,w)\}|- (n/4)| = o(n^3)$.
\item For all $X \subseteq V$, $G' = G[X]$ satisfies
$$
\sum_{v\in X} |d_{G'}^+(v)-d_{G'}^-(v)| = o(n^2);
$$
that is, $G'$ is \emph{almost balanced}. \label{thm:quasi:balanced}
\item Every subtournament $G'$ of $G$ on $\lfloor n/2 \rfloor$ vertices is almost balanced.
\item For every vertex partition $N = X \cup Y$ with $|X| = \lfloor n/2 \rfloor$, $|Y| = \lceil n/2 \rceil$, we have that
$$
\sum_{v \in X} |d^+(v,Y)-d^-(v,Y)| = o(n^2)
$$
\item For all $X,Y \subseteq V$,
$$
\sum_{v \in X} |d^+(v,Y)-d^-(v,Y)| = o(n^2).
$$
\end{enumerate}
\end{theorem}

We say that a tournament that satisfies any of the conditions of Theorem~\ref{thm:quasi} is a \textit{quasi-random} tournament. Our main result on quasi-random tournaments is the following.

\begin{theorem}\label{thm:quasi}
Let $G$ be a quasi-random tournament. If $G$ has diameter $2$ and there exists some constant $0 < \varepsilon < 1$ where $s(x,y) \le \varepsilon n$ for all $x,y \in V$, then
$$
(1+o(1))\log_2 n \le \zeta(G) \le \max(-2/(\log_2 \varepsilon), 2) \log_2 n .
$$
\end{theorem}
\begin{proof}
For the lower bound, fix $0 < \alpha < 1/2$. We show that $\zeta(G) \ge \log_{2}(\alpha n) = (1+o(1))\log_2 n$.
By Theorem~\ref{thm:quasi1}\eqref{thm:quasi:balanced}, we know that for all but at most $o(n)$ vertices, the out-neighborhood has cardinality $(1+o(1))n/2$.
Call this set of vertices $V_0$ and in round 0, place the robber anywhere on $V\setminus V_0$.
Suppose there are $m < \log_2(\alpha n)$ cops.

Since $G$ has diameter $2$, we have that the cardinality of the set of distance vectors $(d_1, d_2, \ldots, d_m)$ for the vertices of $N^+[v]$ is at most $2^m < \alpha n$. However, we have that $|N^+[v] \setminus V_0| = n/2(1+o(1))$, and so the number of distance vectors is strictly less than the size of $N^+[v]\setminus V_0$. Hence, by the pigeonhole principle, there exists at least two vertices of $N^+[v]\setminus V_0$ that have the same distance vector.

As all the vertices of $w \in N^+[v]\setminus V_0$ will also have $|N^+[w] \setminus V_0| = n/2(1+o(1))$, the robber may always evade capture by moving to one of the vertices in $N^+[w]\setminus V_0$ with a duplicate distance vector associated with it.  Thus, $\zeta(G) \ge (1+o(1))\log_2 n$.

Now we turn to the upper bound, where our approach is to bound $\beta(G).$ We say that a vertex $v$ \emph{distinguishes} a pair $x,y \in V$ if $v \in \overline{S}(x,y)$ since then $x$ and $y$ have different distance probes to $v$. By Theorem~\ref{thm:quasi} \eqref{thm:quasi:s}, for all pairs $\{x,y\} \in \binom{V}{2} \setminus X_0$ we have that
$$
\overline{s}(x,y) = (1+o(1))\frac{n}{2} \qquad \textrm{and} \qquad s(x,y) = (1+o(1))\frac{n}{2},
$$
where $|X_0| = o(n^2)$.

Let $W$ be a random set of $k$ vertices in $V$, for
$$k = d \log_2 n \qquad \textrm{and}\qquad d(\varepsilon) =\max(-2/(\log_2 \varepsilon), 2).$$
Note that $2 \le d(\varepsilon) = O(1)$ for all fixed $0<\varepsilon<1$.
For a random choice of $W$, let $p_W$ be the probability that $W$ does not distinguish a pair $\{x,y\} \in \binom{V}{2} \setminus X_0$.
Hence, for pairs $\{x,y\} \in X_0$, by assumption $s(x,y) \le \varepsilon n$. In the case that $\varepsilon \ge 1/2$, we have that
$$
p_W \le \left(\frac{s(x,y)}{n}\right)^k\le \left(\frac{\varepsilon n}{n}\right)^k
= \varepsilon^k
= 2^{\log_2 \varepsilon \cdot d\log_2 n}
= 2^{ -2 \log_2 n}
= \frac1{n^2}.
$$
If $0 < \varepsilon \le 1/2$, then $p_W \le \varepsilon^k \le 1/2^k = 1/2^{2\log_2 n}= 1/n^2$.

Thus, we derive that the expected number of pairs $\{x,y\} \in \binom{V}{2}$ that are not distinguished is at most
$$
\frac{n^2}{2}p_W \le \frac12,
$$
so by the probabilistic method, there exists at least one $W$ that distinguishes all pairs. Hence, we have that $\beta(G) \le k$.
\end{proof}

A tournament $G$ of order $n$ if \textit{doubly regular} if it is a regular tournament such that for all pairs of vertices $x$ and $y$, the number of common out- and in-neighbors of $x$ and $y$ equals $(n-3)/4$. Note that doubly regular tournaments are quasi-random tournaments since they are regular. In addition, $s(x,y) = (n-3)/2$, which satisfies the conditions of Theorem~\ref{thm:quasi}. It was shown in~\cite{reid-brown-1972} that there are doubly regular tournaments of order $n$ if $n +1 = 2^t \prod k_i$, where $t$ is a nonnegative integer and each $k_i \equiv 0 \pmod 4$ is a prime power plus one.

\begin{corollary}\label{drt}
If $G$ is a doubly regular tournament of order $n$, then
$$
(1+o(1))\log_2 n \le \zeta(G) \le \beta(G) \le 2\log_2 n.
$$
\end{corollary}
\begin{proof}
Given the remarks above, it suffices to show that $G$ has diameter $2$. Suppose otherwise for a contradiction. There then exists two vertices $x,y$ with $(y,x) \in E(G)$ and for all $v\in \overline{S}(x,y)$, we have $(x,v), (v,y)\in E(G)$. We then have that $N^+(y) = (n-3)/4 + (n-1)/2 > (n-1)/2$, a contradiction.
Since $s(x,y) \le n/2$, the result follows.
\end{proof}

For a prime power $q,$ the \emph{Paley tournament} of order $q$ has vertices the elements of the finite field of order $q$, with $(i,j)$ an arc if $j-i$ is a quadratic residue in the field. Notice that Paley tournaments are doubly regular, yielding the following application of Corollary~\ref{drt}.
\begin{corollary}
If $P_q$ is the Paley tournament of prime power order $q$, then
$$
(1+o(1))\log_2 q \le \zeta(P_q) \le 2\log_2 q.
$$
\end{corollary}

\section{Further directions}

We introduced the analogue of the Localization game for digraphs, and studied bounds and properties of the localization number. Many questions remain about this digraph parameter. While we proved that digraphs $G$ of order $n$ exist with $\zeta(G) = (1-o(1))n/2$, it remains open to show whether there are digraphs $G$ of order $n$ satisfying $\zeta(G) =\epsilon n$, where $1/2 < \epsilon<1.$

We focused on oriented graphs in this paper, but another direction is to consider more general digraphs that include \emph{digons} (that is, undirected edges). An open problem is to determine the localization number of directed trees that include digons.

We also would like to determine the precise value of $\zeta(T(n,p))$ when $p$ is constant. For example, if $p = 1/2$, then there is a factor of $2$ between our upper and lower bounds. Another question is to investigate $\zeta(T(n,p))$, where $p = p(n)$ is a function of $n$. As was done for localization number in~\cite{dudek-frieze-pegden-2019} on undirected random graphs, it would be interesting to determine $\zeta(T(n,p))$ when $T(n,p)$ has diameter $2$.

\end{document}